\pdfoutput=1
\RequirePackage{pdf14}
\documentclass[11pt]{article} 
\usepackage{amsthm,amssymb} 
\usepackage[leqno]{amsmath}
\newif\ifmicrotype
\relax
\AtBeginDocument{
  \def\tempaa{\thesection.\arabic{equation}}
  \ifx\theequation\tempaa
     \def\theequation{\thesection\hbox{--}\arabic{equation}}
  \fi
  \ifmicrotype
     \RequirePackage{microtype}%
     \DeclareMicrotypeAlias{mnt}{ptm}%
  \fi
  \DeclareMathSizes{10}   {10}   {7}{6}
  \DeclareMathSizes{10.95}{10.95}{8}{6}
}

\DeclareFontFamily{OT1}{ptm}{}
\DeclareFontShape{OT1}{ptm}{m}{n} { <-> ptmr}{}
\DeclareFontShape{OT1}{ptm}{m}{it}{ <-> ptmri}{}
\DeclareFontShape{OT1}{ptm}{m}{sl}{ <->ptmro}{}
\DeclareFontShape{OT1}{ptm}{m}{sc}{ <-> ptmrc}{}
\DeclareFontShape{OT1}{ptm}{b}{n} { <-> ptmb}{}
\DeclareFontShape{OT1}{ptm}{b}{it}{ <-> ptmbi}{}     
\DeclareFontShape{OT1}{ptm}{bx}{n} {<->ssub * ptm/b/n}{}
\DeclareFontShape{OT1}{ptm}{bx}{it}{<->ssub * ptm/b/it}{}

\DeclareSymbolFont{bold}{OT1}{ptm}{b}{n}
\DeclareMathAlphabet{\mathbf}{OT1}{ptm}{b}{n}  
\DeclareMathAlphabet{\mathrm}{OT1}{ptm}{m}{n}

\DeclareFontFamily{OT1}{psy}{}      
\DeclareFontShape{OT1}{psy}{m}{n}{ <-> s * [0.9] psyr}{}
\DeclareFontFamily{OMS}{ptm}{}     
\DeclareFontShape{OMS}{ptm}{m}{n}{ <8> <9> <10> gen * cmsy }{}
\DeclareFontFamily{OMS}{cmtt}{}     
\DeclareFontShape{OMS}{cmtt}{m}{n}{ <8> <9> <10> gen * cmsy }{}

\SetSymbolFont{operators}{normal}{OT1}{ptm}{m}{n}   
\SetSymbolFont{operators}{bold}{OT1}{ptm}{b}{n}     
\DeclareSymbolFont{emsy}{OT1}{ptm}{m}{it}
\DeclareSymbolFont{emsr}{OT1}{ptm}{m}{n}
\DeclareSymbolFont{emcmr}{OT1}{cmr}{m}{n}   
\DeclareSymbolFont{emsymb}{OT1}{psy}{m}{n}  
\DeclareMathSymbol a{\mathalpha}{emsy}{"61}
\DeclareMathSymbol b{\mathalpha}{emsy}{"62}
\DeclareMathSymbol c{\mathalpha}{emsy}{"63}
\DeclareMathSymbol d{\mathalpha}{emsy}{"64}
\DeclareMathSymbol e{\mathalpha}{emsy}{"65}
\DeclareMathSymbol f{\mathalpha}{emsy}{"66}
\DeclareMathSymbol g{\mathalpha}{emsy}{"67}
\DeclareMathSymbol h{\mathalpha}{emsy}{"68}
\DeclareMathSymbol i{\mathalpha}{emsy}{"69}
\DeclareMathSymbol j{\mathalpha}{emsy}{"6A}
\DeclareMathSymbol k{\mathalpha}{emsy}{"6B}
\DeclareMathSymbol l{\mathalpha}{emsy}{"6C}
\DeclareMathSymbol m{\mathalpha}{emsy}{"6D}
\DeclareMathSymbol n{\mathalpha}{emsy}{"6E}
\DeclareMathSymbol o{\mathalpha}{emsy}{"6F}
\DeclareMathSymbol p{\mathalpha}{emsy}{"70}
\DeclareMathSymbol q{\mathalpha}{emsy}{"71}
\DeclareMathSymbol r{\mathalpha}{emsy}{"72}
\DeclareMathSymbol s{\mathalpha}{emsy}{"73}
\DeclareMathSymbol t{\mathalpha}{emsy}{"74}
\DeclareMathSymbol u{\mathalpha}{emsy}{"75}
\DeclareMathSymbol v{\mathalpha}{emsy}{"76}
\DeclareMathSymbol w{\mathalpha}{emsy}{"77}
\DeclareMathSymbol x{\mathalpha}{emsy}{"78}
\DeclareMathSymbol y{\mathalpha}{emsy}{"79}
\DeclareMathSymbol z{\mathalpha}{emsy}{"7A}
\DeclareMathSymbol A{\mathalpha}{emsy}{"41}
\DeclareMathSymbol B{\mathalpha}{emsy}{"42}
\DeclareMathSymbol C{\mathalpha}{emsy}{"43}
\DeclareMathSymbol D{\mathalpha}{emsy}{"44}
\DeclareMathSymbol E{\mathalpha}{emsy}{"45}
\DeclareMathSymbol F{\mathalpha}{emsy}{"46}
\DeclareMathSymbol G{\mathalpha}{emsy}{"47}
\DeclareMathSymbol H{\mathalpha}{emsy}{"48}
\DeclareMathSymbol I{\mathalpha}{emsy}{"49}
\DeclareMathSymbol J{\mathalpha}{emsy}{"4A}
\DeclareMathSymbol K{\mathalpha}{emsy}{"4B}
\DeclareMathSymbol L{\mathalpha}{emsy}{"4C}
\DeclareMathSymbol M{\mathalpha}{emsy}{"4D}
\DeclareMathSymbol N{\mathalpha}{emsy}{"4E}
\DeclareMathSymbol O{\mathalpha}{emsy}{"4F}
\DeclareMathSymbol P{\mathalpha}{emsy}{"50}
\DeclareMathSymbol Q{\mathalpha}{emsy}{"51}
\DeclareMathSymbol R{\mathalpha}{emsy}{"52}
\DeclareMathSymbol S{\mathalpha}{emsy}{"53}
\DeclareMathSymbol T{\mathalpha}{emsy}{"54}
\DeclareMathSymbol U{\mathalpha}{emsy}{"55}
\DeclareMathSymbol V{\mathalpha}{emsy}{"56}
\DeclareMathSymbol W{\mathalpha}{emsy}{"57}
\DeclareMathSymbol X{\mathalpha}{emsy}{"58}
\DeclareMathSymbol Y{\mathalpha}{emsy}{"59}
\DeclareMathSymbol Z{\mathalpha}{emsy}{"5A}
\DeclareMathSymbol{\bullet}{\mathalpha}{emsymb}{"B7}
\DeclareMathSymbol{\regis}{\mathalpha}{emsymb}{"D2}
\def\Bullet{\leavevmode\unkern{$\m@th\bullet$}\kern.32em\ignorespaces}
\def\Regis{\leavevmode\raise.5ex\hbox{$\m@th\regis$}}
\DeclareMathSymbol +{\mathbin}{emcmr}{`+}
\DeclareMathSymbol ={\mathrel}{emcmr}{`=}  
\DeclareMathSymbol{\Gamma}{\mathalpha}{emcmr}{"00}
\DeclareMathSymbol{\Delta}{\mathalpha}{emcmr}{"01}
\DeclareMathSymbol{\Theta}{\mathalpha}{emcmr}{"02}
\DeclareMathSymbol{\Lambda}{\mathalpha}{emcmr}{"03}
\DeclareMathSymbol{\Xi}{\mathalpha}{emcmr}{"04}
\DeclareMathSymbol{\Pi}{\mathalpha}{emcmr}{"05}
\DeclareMathSymbol{\Sigma}{\mathalpha}{emcmr}{"06}
\DeclareMathSymbol{\Upsilon}{\mathalpha}{emcmr}{"07}
\DeclareMathSymbol{\Phi}{\mathalpha}{emcmr}{"08}
\DeclareMathSymbol{\Psi}{\mathalpha}{emcmr}{"09}
\DeclareMathSymbol{\Omega}{\mathalpha}{emcmr}{"0A}
\DeclareMathSizes{7.6}{8}{6}{5}

\def\`#1{{\accent"12 #1}}            
            
\chardef\J="11                  
                   
\chardef\AA="C8                      
\chardef\gbp="A3                    
\chardef\TIL="81                     
\chardef\endash="B1                
\chardef\emdash="D0                
\chardef\pourmille="BD              
\chardef\aoben="E3                   
\chardef\ooben="EB                  
\def\S{\leavevmode\unkern{\char"A7}\kern.1em\ignorespaces} 
            
\DeclareMathAccent{\dot}{\mathalpha}{operators}{"C7} 
\def\og{{\char"AB}\kern.15em}
\def\fg{\relax\ifhmode\unskip\kern.15em\fi{\char"BB}}

\def\cedpol#1{\setbox0=\hbox{#1}\ifdim\ht0=1ex \accent"CE #1%
  \else{\ooalign{\hidewidth\char"CE\hidewidth\crcr\unhbox0}}\fi}

\newskip\stdskip                      % 
\newcommand{\stdspace}{\hskip 0.75em plus 0.15em \ignorespaces}

\newtheoremstyle{plain}{13.2pt plus6.6pt minus6.6pt}{6.6pt plus3.3pt minus3.3pt}%
{\sl}{}{\bf}{}{0.75em}{\thmname{#1}\thmnumber{ #2}\thmnote{\rm\stdspace(#3)}}

\newtheoremstyle{definition}{13.2pt plus6.6pt minus6.6pt}{6.6pt plus3.3pt minus3.3pt}%
{\rm}{}{\bf}{}{0.75em}{\thmname{#1}\thmnumber{ #2}\thmnote{\rm\stdspace(#3)}}

\newtheoremstyle{remark}{13.2pt plus6.6pt minus6.6pt}{6.6pt plus3.3pt minus3.3pt}%
{\rm}{}{\bf}{}{0.75em}{\thmname{#1}\thmnumber{ #2}\thmnote{\rm\stdspace(#3)}}

\theoremstyle{plain}

\newcommand{\co}{\mskip0.5mu\colon\thinspace}

\usepackage{geometry}

    \newgeometry{vmargin={0mm, 37mm}, hmargin={30mm,23mm}}   % set the margins
\usepackage{tikz-cd}
\usepackage{xparse}
\usepackage[most]{tcolorbox}
\usepackage{amsthm}

\RequirePackage[bookmarks, bookmarksopen=true, plainpages=false, pdfpagelabels, pdfpagelayout=SinglePage, breaklinks = true]{hyperref}
\usepackage[nameinlink, noabbrev,capitalize]{cleveref}
\usepackage{xcolor}
\hypersetup{
    colorlinks,
    linkcolor={red!50!black},
    citecolor={blue},
    urlcolor={red!80!black}
}

\usepackage{savesym}
\savesymbol{checkmark}
\usepackage{dingbat}

\makeatletter
\newcommand{\tpitchfork}{%
  \vbox{
    \baselineskip\z@skip
    \lineskip-.52ex
    \lineskiplimit\maxdimen
    \m@th
    \ialign{##\crcr\hidewidth\smash{$-$}\hidewidth\crcr$\pitchfork$\crcr}
  }%
}
\makeatother

\newtheorem*{namedthm}{\namedthmname}
\newcounter{namedthm}



\usepackage{amssymb} % or amsfonts

\newtheorem{theorem}{Theorem}[section]
\newtheorem{lemma}[theorem]{Lemma}

\newtheorem{remark}[theorem]{Remark}

\newtheorem*{def*}{Definition}
\newtheorem*{rem*}{Remark}

\usepackage[numbered]{bookmark}

\makeatletter
\newcommand*{\sectionbookmark}[1][]{%
  \bookmark[%
    level=section,%
    dest=\@currentHref,%
    #1%
  ]%
}
\makeatother
\usepackage{etoolbox}
\makeatletter
\if@twoside
   \patchcmd{\ps@headings}%
      {\@evenhead{\thepage\hfil\slshape\leftmark}}%
      {\@evenhead{\thepage\hfil}}{}{}
\fi
   \patchcmd{\ps@headings}%
      {\@oddhead{{\slshape\rightmark}\hfil\thepage}}%
      {\@oddhead{\hfil\thepage}}{}{}  
\makeatother

\usepackage{mathrsfs}
\usepackage{adjustbox}
\usepackage[figurename=Fig.]{caption}
\usepackage{float}
\usepackage{graphicx}
\usepackage{amssymb}

\usepackage{tikz}

\DeclareFontFamily{U}{mathb}{\hyphenchar\font45}
\DeclareFontShape{U}{mathb}{m}{n}{%
  <-6> mathb5
  <6-7> mathb6
  <7-8> mathb7
  <8-9> mathb8
  <9-10> mathb9
  <10-12> mathb10
  <12-> mathb12
  } {}%
\DeclareSymbolFont{mathb}{U}{mathb}{m}{n}

\DeclareMathSymbol{\abxpitchfork}{\mathord}{mathb}{"26}

\usepackage{stackengine}

\usepackage{mathtools}

\author{Sumanta Das and Siddhartha Gadgil}
\date{}
\title{Surfaces of infinite-type are non-Hopfian}

\usepackage{thmtools}

\newtheorem{theoaux}[subsection]{Theorem}

\NewDocumentEnvironment{theo}{o}
 {\IfNoValueTF{#1}
   {\theoaux\addcontentsline{toc}{subsection}{Theorem \protect\numberline{\thesubsection}}}
   {\theoaux[#1]\addcontentsline{toc}{subsection}{Theorem \protect\numberline{\thesubsection} (#1)}}%
   \ignorespaces}
 {\endtheoaux}

\usepackage[shortlabels]{enumitem} 

\usepackage{varwidth}
\usepackage{bbding}
\usepackage[nottoc,notlot,notlof]{tocbibind}
\definecolor{green}{rgb}{0.0, 0.5, 0.0}
\begin{document}

\maketitle

\begin{abstract}\mbox{}\\We show that finite-type surfaces are characterized by a topological analog of the Hopf property.
Namely, an oriented surface $\Sigma$ is of finite-type if and only if every proper map $f\co\Sigma\to \Sigma$ of degree one is homotopic to a homeomorphism. \end{abstract}

\section{Introduction}

All surfaces will be assumed to be connected and orientable throughout this note. We will say a surface is of \emph{finite-type} if its fundamental group is finitely generated; otherwise, we will say it is of \emph{infinite-type}.

Recall that a group $G$ is said to be \emph{Hopfian}  if every surjective homomorphism $\varphi\co G\twoheadrightarrow G$ is an isomorphism. It is well known that a finitely generated free group is Hopfian, for instance, as a consequence of Grushko's theorem. On the other hand, a free group generated by an infinite set $S$ is not Hopfian as a surjective function $f\co S \to S$ that is not injective extends to a surjective homomorphism on the free group generated by $S$ which is not injective.

In this note, we show that there is an analogous characterization for orientable surfaces of \emph{finite-type}. The natural topological analog of a surjective group homomorphism is a proper map of  degree one, and that of an isomorphism is a homotopy equivalence.

One-half of this characterization is classical, namely that any proper map of degree one from a surface of finite-type to itself is a homotopy equivalence. For instance, a theorem of Olum (see~\cite[Corollary 3.4]{MR192475}) says that every proper map of degree one between two oriented manifolds of the same dimension is $\pi_1$-surjective. Now, the fundamental group of any surface is  residually finite (see \cite{MR295352}). Also, any finitely generated residually finite group is Hopfian. Thus, every degree one self map of a finite-type surface is a weak homotopy equivalence, hence a homotopy equivalence by Whitehead's theorem.

Our main result is that infinite-type surfaces are not Hopfian.

\begin{theo}\label{main}\textup{Let $\Sigma$ be any infinite-type surface. Then there exists a proper map $f\co \Sigma\to \Sigma$ of degree one such that $\pi_1(f)\co \pi_1(\Sigma)\to \pi_1(\Sigma)$ is not injective. In particular, $f$ is not a homotopy equivalence.}
 \end{theo}
 
\section{Background}A \emph{surface} is a connected, orientable two-dimensional manifold without boundary, and a \emph{bordered surface} is a connected, orientable two-dimensional manifold with a non-empty boundary. A (possibly bordered) subsurface $\Sigma'$ of a surface $\Sigma$ is an embedded submanifold of codimension zero.

Let $\Sigma$ be a non-compact surface.
A \emph{boundary component} of $\Sigma$ is a nested sequence $P_1\supseteq P_2\supseteq \cdots$ of open, connected subsets of $\Sigma$ such that the followings hold:
\begin{itemize}
    \item the closure (in $\Sigma$) of each $P_n$ is non-compact,
    \item the boundary of each $P_n$ is compact, and 
    \item for any subset $A$ with compact closure (in $\Sigma$),  we have $P_n\cap A=\varnothing$ for all large $n$.
\end{itemize}

We say that two boundary components $P_1\supseteq P_2\supseteq \cdots$ and $P_1'\supseteq P_2'\supseteq \cdots$ of $\Sigma$ are \emph{equivalent} if for any positive integer $n$ there are positive integers $k_n,\ell_n$ such that $P_{k_n}\subseteq P'_n$ and $P'_{\ell_n}\subseteq P_n$. 
For a boundary component $\mathscr P=P_1\supseteq P_2\supseteq \cdots$, we  let $[\mathscr P]$ to denote the equivalence class of $\mathscr P$.

The \emph{space of ends} $\textup{Ends} (\Sigma)$ of $\Sigma$ is the topological space having equivalence classes of boundary components of $\Sigma$ as elements, i.e., as a set $\textup{Ends}(\Sigma)\coloneqq\big\{[\mathscr P]\big| \mathscr P\text{ is a boundary component }\big\};$ with the following topology: For any set $X$ with compact boundary, at first, define $$X^\dag\coloneqq \big\{[\mathscr P=P_1\supseteq P_2\supseteq\cdots]\big| X\supseteq P_n\supseteq P_{n+1}\supseteq\cdots\text{ for some large }n\big\}.$$

Now, take the set of all such $X^\dag$ as a basis for the topology of $\textup{Ends}(\Sigma)$. The topological space $\textup{Ends}(\Sigma)$ is compact, separable, totally disconnected, and metrizable, i.e., homeomorphic to a non-empty closed subset of the Cantor set. 

For a boundary component $[\mathscr P]$ with $\mathscr P=P_1\supseteq P_2\supseteq\cdots$, we say $[\mathscr P]$ is \emph{planar} if $P_n$ are homeomorphic to open subsets $\mathbb R^2$ for all large $n$. Define $\textup{Ends}_\text{np}(\Sigma)\coloneqq\big\{[\mathscr P]: [\mathscr P]\text{ is }\text{ not }\text{ planar}\big\}$. Thus, $\textup{Ends}_\text{np}(\Sigma)$ is a closed subset of $\textup{Ends}(\Sigma)$. Also, define the \emph{genus} of $\Sigma$ as $g(\Sigma)\coloneqq\sup g(\mathbf S)$,  where $\mathbf S$ is a compact bordered subsurface of $\Sigma$.

\begin{theorem}{\textup{\big(Kerékjártó's classification theorem \cite[Theorem 1]{MR143186}\big)}}
 \textup{Let $\Sigma_1,\Sigma_2$ be two non-compact surfaces. Then $\Sigma_1$ is homeomorphic to $\Sigma_2$ if and only if  $g(\Sigma_1)=g(\Sigma_2)$, and  there is a homeomorphism $\Phi\co \textup{Ends}(\Sigma_1)\to \textup{Ends}(\Sigma_2)$ with $\Phi\big(\textup{Ends}_\textup{np}(\Sigma_1)\big)=\textup{Ends}_\textup{np}(\Sigma_2)$.}\label{richard2}
\end{theorem}
Let $\Sigma$ be a non-compact surface, and let $\mathscr E_\text{np}(\Sigma)\subseteq \mathscr E(\Sigma)$ be two closed, totally-disconnected subsets of $\mathbb S^2$ such that the pair  $\text{Ends}_\text{np}(\Sigma)\subseteq \text{Ends}(\Sigma)$ is homeomorphic to the pair $\mathscr E_\text{np}(\Sigma)\subseteq \mathscr E(\Sigma)$. Consider a pairwise disjoint collection $\{D_i\subseteq \mathbb S^2\setminus \mathscr E(\Sigma): i\in \mathscr A\}$ of closed disks, where $|\mathscr A|= g(\Sigma)$, such that the following holds: For $p\in \mathbb S^2$, any open neighborhood (in $\mathbb S^2$) of $p$ contains infinitely many $D_i$ if and only if $p\in \mathscr E_\textup{np}(\Sigma)$. The proof of \cite[Theorem 2]{MR143186} describes constructing such a collection of disks.

Now, let $M\coloneqq (\mathbb S^2\setminus \mathscr E(\Sigma))\setminus\sqcup_{i\in \mathscr A}\text{int}(D_i)$ and $N\coloneqq \sqcup_{i\in \mathscr A} S_{1,1}$, where $S_{1,1}$ is the genus one compact bordered surface with one boundary component. Define a non-compact surface $\Sigma_\text{handle}$ as follows: $\Sigma_\text{handle}\coloneqq M\sqcup_{\partial M\equiv \partial N} N$. Then we have the following theorem.

\begin{theorem}{\textup{\big(Richards' representation theorem \cite[Theorems 2 and 3]{MR143186}\big)}}
\textup{The surface $\Sigma_\textup{handle}$ is homeomorphic to $\Sigma$.} \label{richard1}\end{theorem}

\section{Proof of \texorpdfstring{\Cref{main}}{Theorem \ref{main}}}
Let $M$ and $N$ be two non-compact, oriented, connected, boundaryless smooth  $n$-manifolds. Then the singular cohomology groups with compact support $H^n_\textbf{c}(M;\mathbb Z)$ and $H^n_\textbf{c}(N;\mathbb Z)$ are infinite cyclic with preferred generators $[M]$ and $[N]$. If $f\colon M\to N$ is a proper map then the  degree of $f$ is the unique integer $\deg(f)$ defined as follows: $H^n_\textbf{c}(f)([N])=\deg(f)\cdot [M]$.  Note that $\deg$ is proper-homotopy invariant and multiplicative. See \cite[Section 1]{MR192475} for more details.

We will use the following well-known characterization of degree.

\begin{lemma}{\textup{\cite[Lemma 2.1b.]{MR192475}}}
\textup{Let $f\co M\to N$ be a proper map between two non-compact, oriented, connected, boundaryless smooth  $n$-manifolds. Let $D$ be a smoothly embedded closed disk in $N$ and suppose $f^{-1}(D)$ is a smoothly embedded closed disk in $M$ such that $f$ maps $f^{-1}(D)$ homeomorphically onto $D$. Then $\deg (f)=+1$ or $-1$ according as $f\vert f^{-1}(D)\to D$ is orientation-preserving or orientation-reversing.}
\label{degreeonemapchecking}\end{lemma}

We will prove \Cref{main} by considering the following three cases: 
\begin{enumerate}
  \item $\Sigma$ has infinite genus.
  \item $\Sigma$ has finite genus and the set of isolated points $\mathscr I(\Sigma)$ of $\mathscr E(\Sigma)$ is finite.
  \item $\Sigma$ has finite genus and the set of isolated points $\mathscr I(\Sigma)$ of $\mathscr E(\Sigma)$ is infinite.
\end{enumerate} 

\begin{remark}
\textup{If $\Sigma$ is an infinite-type surface of a finite genus, then $\mathscr E(\Sigma)$ is an infinite set.}\label{infinitetypefinitegenusmeansnumberofendsisinfinite}
\end{remark}

Our first result proves \Cref{main} in the case with infinite genus.

\begin{theorem}
    \textup{Let $\Sigma$ be a surface of the infinite genus. Then there exists a degree one map $f\co \Sigma\to \Sigma$ which is not $\pi_1$-injective.} \label{infinitegenus}
\end{theorem}
\begin{proof}
Since $\Sigma$ has infinite genus, there exists a compact bordered subsurface $\mathcal S\subset \Sigma$ such that $\mathcal S$ has genus one and one boundary component. Define $\Sigma'$ as $\Sigma'\coloneqq \Sigma/\mathcal S$, i.e., $\Sigma'$ is the quotient of $\Sigma$ with $\mathcal S$ pinched to a point. Let $q\co\Sigma \to \Sigma'$ be the quotient map. Thus, $\Sigma'$ is also an infinite genus surface. Further, there are compact sets in $K\subset \Sigma$ and $K' \subset\Sigma'$ whose complements are homeomorphic, so the pair $(\mathscr E(\Sigma), \mathscr E_\text{np}(\Sigma))$ is homeomorphic to the pair $(\mathscr E(\Sigma'), \mathscr E_\text{np}(\Sigma'))$. Hence, by \Cref{richard2}, there is a homeomorphism  $\varphi\co\Sigma'\to\Sigma$. 

Let $f \co \Sigma \to \Sigma$ be the composition $f = \varphi\circ q$. 
By \Cref{degreeonemapchecking}, the quotient map $q\co\Sigma\to \Sigma'$ is of degree $\pm 1$. Thus, $\deg(f)=\pm 1$ as homeomorphisms have degree $\pm 1$. Notice that $f$ sends $\partial \mathcal S$ to a point. But $\partial \mathcal S$ does not bound any disk in $\Sigma$, i.e., $\partial \mathcal S$ represents a primitive element of $\pi_1(\Sigma)$, see \cite[Theorem 1.7. and Theorem 4.2.]{MR214087}. Hence, $f$ is not $\pi_1$-injective. If $\deg(f) = 1$, then we are done. Otherwise, we replace $f$ by $f \circ f$ to get a map that has degree one and is not injective on $\pi_1$.
\end{proof}

For the remaining two cases, we use a map from the sphere to the sphere, which has degree $\pm 1$ but with some disks identified. We will replace these disks with appropriate surfaces to get $\Sigma$.

\begin{lemma}
\textup{There exist pairwise disjoint closed disks $\mathcal D_0, \mathcal D_1\subseteq \mathbb S^2$ and a map $f\co \mathbb S^2\to \mathbb S^2$  such that the following hold: }
\begin{itemize}
    \item \textup{$f^{-1}(\mathcal D_0)=\mathcal D_0$ and $f\vert_{\mathcal D_0}\co \mathcal{D}_0\to \mathcal D_0$ is the identity map.}
    \item \textup{$f^{-1}(\mathcal D_1)$ is the union of pairwise-disjoint closed disks $\mathcal D_{1,1}, \mathcal D_{1,2}$, and $\mathcal D_{1,3}$ in $\mathbb S^2$; and $f\vert_{\mathcal{D}_{1k}}\co \mathcal{D}_{1k}\to \mathcal D_{1}$ is a homeomorphism for each $k\in \{1,2,3\}$. }
    \label{foldingmapconstruction} 
\end{itemize}
\textup{Further, there is a loop $\gamma$ in $\mathbb S^2\setminus \textup{int}(\mathcal D_0\cup \mathcal D_{1, 1}\cup \mathcal D_{1,2}\cup \mathcal D_{1,3})$ which is not homotopically trivial in $\mathbb S^2\setminus \textup{int}(\mathcal D_0\cup \mathcal D_{1, 1}\cup \mathcal D_{1,2}\cup \mathcal D_{1,3})$, but such that $f(\gamma)$ is null-homotopic in $\mathbb S^2\setminus\textup{int} (\mathcal D_0\cup \mathcal D_1)$.} \label{foldingmap}
\end{lemma}

\begin{proof}
For each $k\in \{0,1,2,3\}$, choose $(a_k, b_k)\in \mathbb R^2$ such that if we define $$\mathcal B_k\coloneqq\left\{(x,y)\in \mathbb R^2 : (x-a_k)^2+(y-b_k)^2\leq 1\right\},$$ then $\{\mathcal B_0, \mathcal B_1, \mathcal B_2, \mathcal B_3\}$ is a pairwise-disjoint collection of closed disks. 

Define $X\coloneqq\mathbb S^2\setminus\bigcup_{i=0}^3\text{int}(\mathcal B_i)$ and $Y\coloneqq\mathbb S^2\setminus\bigcup_{i=0}^1\text{int}(\mathcal B_i)$. Next, define a map $f\colon \partial X\to Y$ as follows:
\begin{itemize}
    \item  $f\vert_{\partial \mathcal B_k} \co \partial \mathcal B_k\to \partial \mathcal  B_k$ is the identity map for each $k\in\{0,1\}$;
    \item $f\vert_{\partial \mathcal B_2} \co \partial \mathcal B_2 \to \partial \mathcal  B_1$ is defined as $f(x,y)\coloneqq (-x+a_2+a_1, y-b_2+b_1)$ for all $(x,y)\in \partial \mathcal B_2$.
    
    \item $f\vert_{\partial \mathcal B_3} \co \partial \mathcal B_3\to \partial \mathcal  B_1$ is defined as $f(x,y)\coloneqq (x-a_3+a_1, y-b_3+b_1)$ for all $(x,y)\in \partial \mathcal B_3$.
\end{itemize}
\begin{figure}[ht]$$\adjustbox{trim={0.0\width} {0.0\height} {0.0\width} {0.0\height},clip}{\def\svgwidth{1\linewidth}
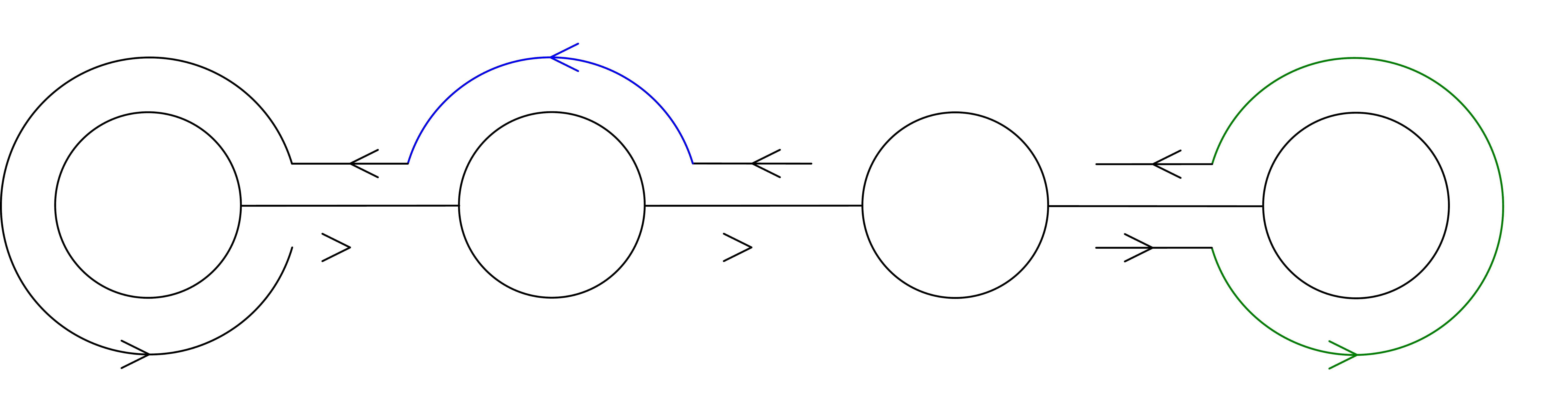
}$$
\caption{The four-holed sphere $X$ by attaching a $2$-cell.}\label{attaching}
\end{figure}

For each $k\in \{0,1,2\}$, let $\gamma_k\co[0,1]\hookrightarrow X$ be an embedding such that $\text{im}(\gamma_k)\cap \partial X$ consists of $\gamma_k(0)=(a_k+1,b_k)\in \partial\mathcal B_k$ and  $\gamma_k(1)=(a_{k+1}-1,b_{k+1})\in \partial\mathcal B_{k+1}$.

Define $\Gamma_0\co [0,1]\to Y$ as $\Gamma_0(t)\coloneqq\gamma_0(t)$ for all $t\in [0,1]$. Let $\Gamma_1, \Gamma_2\colon [0,1]\to Y$ be the constant loops based at the points $(a_1+1,b_1)\in  \partial Y$ and $(a_1-1,b_1)\in  \partial Y$, respectively.

Next, define $X^{(1)}\coloneqq \partial X\cup \text{im}(\gamma_0)\cup \text{im}(\gamma_1)\cup\text{im}(\gamma_2)$. Extend $f\colon \partial X\to Y$ to a map $X^{(1)}\to Y$, which we again denote by $f\co X^{(1)}\to Y$, by mapping $\gamma_0$ onto $\Gamma_0$ by the identity, and, for each $k=1,2$, mapping $\gamma_k$ to the constant loop $\Gamma_k$.

Let $\theta_0$ (resp. $\theta_3$) be the simple loop that traverses $\partial \mathcal B_0$ (resp. $\partial \mathcal B_3$) in the counter-clockwise direction starting from $(a_0+1,b_0)$ (resp. $(a_3-1,b_3)$).

Let $\theta_{1,l}$ (resp. $\theta_{1,u}$) be the simple arc that traverses $\partial \mathcal B_1\cap \{y\leq b_1\}$ (resp. $\partial \mathcal B_1\cap \{y\geq b_1\}$) counter-clockwise direction. Similarly, define $\theta_{2, l}$ and $\theta_{2,u}$.

Now,  $X\cong X^{(1)}\cup_\varphi \mathbb D^2$, (see \Cref{attaching})  where the attaching map $\varphi\colon \mathbb S^1\to X^{(1)}$ can be described as $$\varphi\coloneqq \theta_0*\gamma_0*\theta_{1,l}*\gamma_1*\theta_{2,l}*\gamma_2*\theta_3*\overline \gamma_2*\theta_{2,u}*\overline \gamma_1*\theta_{1,u}*\overline{\gamma_0}.$$

Notice that $f(\gamma_1)=\Gamma_1$ and $f(\gamma_2)=\Gamma_2$ are constant loops. Also, as in \Cref{map-extension},  $\overline{f\circ\theta_{1,l}}=f\circ \theta_{2,l}$ and $\overline{f\circ\theta_{1,u}}=f\circ \theta_{2,u}$. Thus, $f\circ\varphi$ is homotopic to $(f\circ\theta_0)*\Gamma_0*(f\circ\theta_3)*\overline{\Gamma_0}$. 

\begin{figure}[ht]$$\adjustbox{trim={0.0\width} {0.0\height} {0.0\width} {0.0\height},clip}{\def\svgwidth{1\linewidth}
%% Creator: Inkscape 1.1.2 (b8e25be833, 2022-02-05), www.inkscape.org
%% Accompanies image file 'flippingmap2.pdf' (pdf, eps, ps)
\begingroup%
  \makeatletter%
  \providecommand\color[2][]{%
    \errmessage{(Inkscape) Color is used for the text in Inkscape, but the package 'color.sty' is not loaded}%
    \renewcommand\color[2][]{}%
  }%
  \providecommand\transparent[1]{%
    \errmessage{(Inkscape) Transparency is used (non-zero) for the text in Inkscape, but the package 'transparent.sty' is not loaded}%
    \renewcommand\transparent[1]{}%
  }%
  \providecommand\rotatebox[2]{#2}%
  \newcommand*\fsize{\dimexpr\f@size pt\relax}%
  \newcommand*\lineheight[1]{\fontsize{\fsize}{#1\fsize}\selectfont}%
  \ifx\svgwidth\undefined%
    \setlength{\unitlength}{1090.21350407bp}%
    \ifx\svgscale\undefined%
      \relax%
    \else%
      \setlength{\unitlength}{\unitlength * \real{\svgscale}}%
    \fi%
  \else%
    \setlength{\unitlength}{\svgwidth}%
  \fi%
  \global\let\svgwidth\undefined%
  \global\let\svgscale\undefined%
  \makeatother%
  \begin{picture}(1,0.48895777)%
    \lineheight{1}%
    \setlength\tabcolsep{0pt}%
    \put(0,0){\includegraphics[width=\unitlength,page=1]{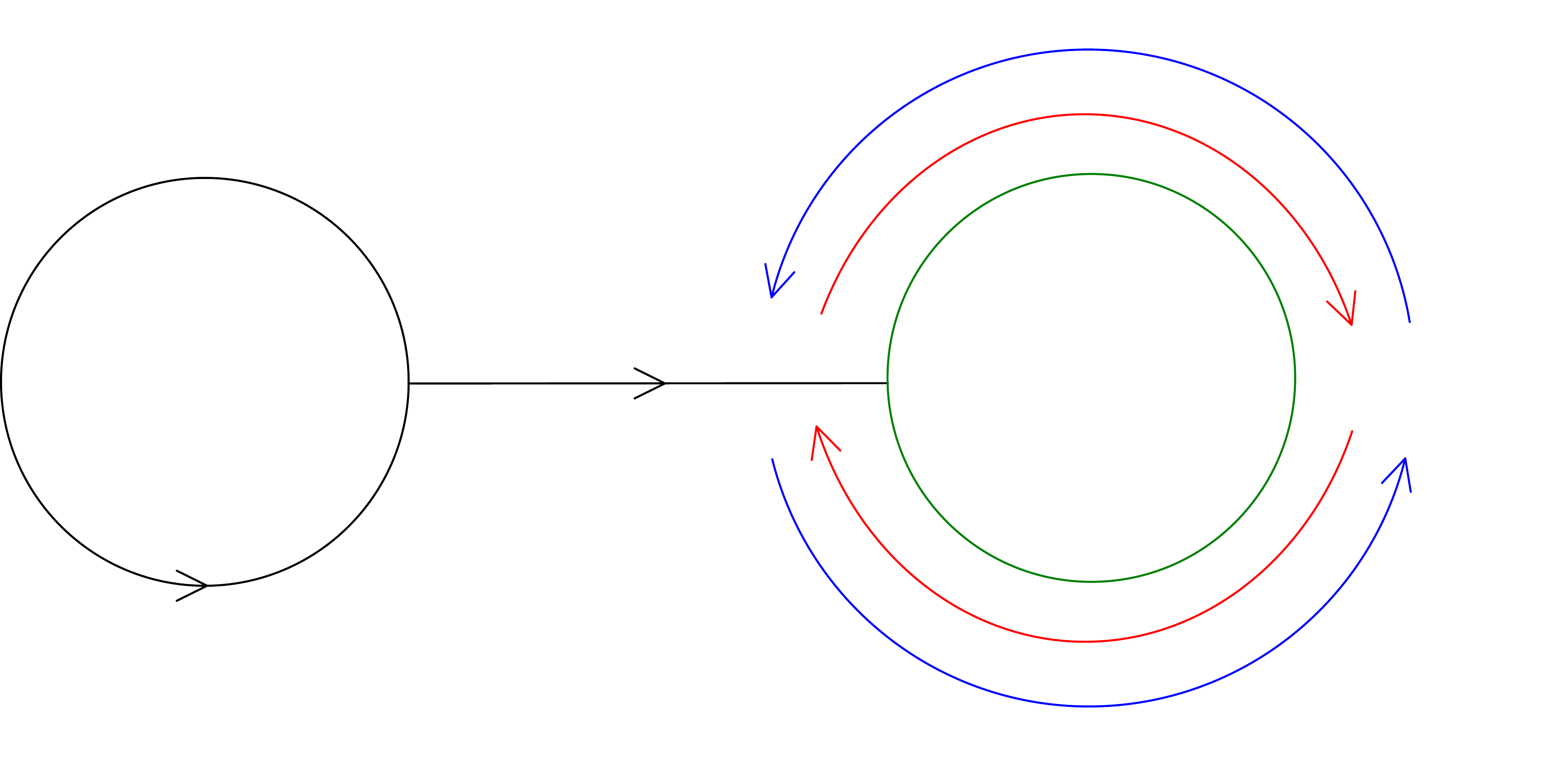}}%
    \put(0.12944729,0.08961527){\makebox(0,0)[lt]{\lineheight{0}\smash{\begin{tabular}[t]{l}$f\circ\theta_0$\end{tabular}}}}%
    \put(0.65897446,0.09217198){\makebox(0,0)[lt]{\lineheight{0}\smash{\begin{tabular}[t]{l}$\color{green}f\circ\theta_3$\end{tabular}}}}%
    \put(0,0){\includegraphics[width=\unitlength,page=2]{flippingmap2.pdf}}%
    \put(0.65897446,0.05227109){\makebox(0,0)[lt]{\lineheight{0}\smash{\begin{tabular}[t]{l}$\color{red}f\circ\theta_{2l}$\end{tabular}}}}%
    \put(0.65897446,0.00961824){\makebox(0,0)[lt]{\lineheight{0}\smash{\begin{tabular}[t]{l}$\color{blue}f\circ\theta_{1l}$\end{tabular}}}}%
    \put(0.65897446,0.42650598){\makebox(0,0)[lt]{\lineheight{0}\smash{\begin{tabular}[t]{l}$\color{red}f\circ\theta_{2u}$\end{tabular}}}}%
    \put(0.65897446,0.46778222){\makebox(0,0)[lt]{\lineheight{0}\smash{\begin{tabular}[t]{l}$\color{blue}f\circ\theta_{1u}$\end{tabular}}}}%
    \put(0.42001426,0.25709107){\makebox(0,0)[lt]{\lineheight{0}\smash{\begin{tabular}[t]{l}$\Gamma_0$\end{tabular}}}}%
  \end{picture}%
\endgroup%

}$$
\caption{The map on the $X^{(1)}$}\label{map-extension}

\end{figure}

If $r\co Y\cong \mathbb S^1\times [0,1]\to \mathbb S^1$ is the  projection  then $r\circ f\circ\theta_0$ and $r\circ f\circ\theta_3$ traverse $\mathbb S^1$ in opposite directions. Since $r$ is a strong deformation retract, $(f\circ\theta_0)*\Gamma_0*(f\circ\theta_3)*\overline{\Gamma_0}$, and hence $f\circ\varphi$ is null-homotopic. Now, the null-homotopic map $f\circ\varphi \colon \mathbb S^1\to Y$ can be extended to a map $\mathbb D^2\to Y$. Thus $f\co X^{(1)}\to Y$ can be extended to a map $X\cong X^{(1)}\cup_\varphi\mathbb D^2\to Y$, which will be again denoted by $f\co X\to Y$. 

Note that every homeomorphism $\mathbb S^1\to \mathbb S^1$ can be extended to a homeomorphism $\mathbb D^2\to \mathbb D^2$ naturally. Thus, we can extend $f\co X\to Y$ to a map $\mathbb S^2\to \mathbb S^2$, which will be again denoted by $f\co \mathbb S^2\to \mathbb S^2$. Let $\mathcal D_0$ (resp. $\mathcal D_1$) be any closed disk, which is contained in $\text{int}(\mathcal B_0)$ (resp. $\text{int}(\mathcal B_1)$).

Finally, observe that if $\gamma=\theta_{1u}*\theta_{1l}*\gamma_1*\theta_{2l}*\theta_{2u}*\overline{\gamma_1}$, then $\gamma$ is a loop in $\mathbb S^2\setminus \textup{int}(\mathcal D_0\cup \mathcal D_{1, 1}\cup \mathcal D_{1,2}\cup \mathcal D_{1,3})$ which is not homotopically trivial in $\mathbb S^2\setminus \textup{int}(\mathcal D_0\cup \mathcal D_{1, 1}\cup \mathcal D_{1,2}\cup \mathcal D_{1,3})$, but $f(\gamma)$ is null-homotopic in $\mathbb S^2\setminus \textup{int}(\mathcal D_0\cup \mathcal D_1)$, as claimed.
\end{proof}

% \begin{remark}
% The (intermediate) map $f\co X\to Y$ in the proof of \Cref{foldingmapconstruction}, sends the non-trivial loop $\theta_{1u}*\theta_{1l}*\gamma_1*\theta_{2l}*\theta_{2u}*\overline{\gamma_1}$ in $X$ to a trivial loop in $Y$. Also, any simple loop $\mathcal L$ in $\mathbb S^2$ with $\mathcal B_1\cup \mathcal B_2\subseteq \textup{interior}(\mathcal L)$ and $\mathcal B_0\cup \mathcal B_3\subseteq \textup{exterior}(\mathcal L)$ represents a non-trivial simple loop in $X$, which goes to a trivial loop of $Y$ by the map $f\co X\to Y$. \label{nonpi_1injective}
% \end{remark}

We now prove \Cref{main} in the two remaining cases, in both of which we have a finite genus surface. Note that for a finite genus surface, all ends are planar, so in applying \Cref{richard2}, it suffices to consider the genus and the space of ends.

\begin{theorem}
    \textup{Let $\Sigma$ be a finite genus infinite-type surface  such that $\mathscr E(\Sigma)$ has finitely many isolated points. Then there is a degree one map $f\co \Sigma\to \Sigma$ which is not $\pi_1$-injective.}\label{finitegenswithfinitelymanyisolatedends}
\end{theorem}

\begin{proof}
Let $\mathscr I(\Sigma)$ be the set of all isolated points of $\mathscr E(\Sigma)$, let $k\in\mathbb{N}\cup\{0\}$ be the cardinality of $\mathscr{I}(\Sigma)$, and let $g$ be the genus of $\Sigma$. Then $\mathscr C(\Sigma)\coloneqq\mathscr E(\Sigma)\setminus\mathscr I(\Sigma)$ is a non-empty, perfect, compact, totally-disconnected, metrizable space as it is infinite (by \Cref{infinitetypefinitegenusmeansnumberofendsisinfinite}) and has no isolated points. Thus $\mathscr C(\Sigma)$ is a Cantor space (see \cite[Theorem 8 of Chapter 12]{MR0488059}). 

Let $\mathcal D_0, \mathcal D_1, \mathcal D_{1,1},\mathcal D_{1, 2},\mathcal D_{1, 3}\subseteq \mathbb S^2$, $f\co\mathbb S^2\to \mathbb S^2$, and $\gamma$ be as in the conclusion of \Cref{foldingmapconstruction}. Also, let $C_1\subset \text{int}(\mathcal D_1)$ be a subset homeomorphic to the Cantor set, and let $\mathscr I\subset \text{int}(\mathcal D_0)$ be a set consisting of $k$ points (hence homeomorphic to $\mathscr I(\Sigma)$). Now, define $C_{1, j}$ as $C_{1, j} \coloneqq f^{-1}(C_1)\cap \mathcal D_{1, j}$ for $j=1,2,3$. Note that each  $C_{1, j}$ is homeomorphic to the Cantor set. See \Cref{cantortree}.

As $f^{-1}(\mathcal D_0)=\mathcal D_0$ and $f\vert_{\mathcal D_0}\co \mathcal{D}_0\to \mathcal D_0$ is the identity map, we can say that $f^{-1}(\mathscr I) = \mathscr{I}$.
Let $\Sigma_1$ be the surface obtained from $\mathbb S^2\setminus (\mathscr I\cup C_1)$ by attaching $g$ handles along disjoint disks $\Delta_{k}\subset \text{int}(\mathcal D_0)\setminus \mathscr I$, $1\leq k \leq g$, and let $\Sigma_2$ be the surface obtained from $\mathbb S^2\setminus (\mathscr I\cup C_{1, 1}\cup C_{1, 2}\cup C_{1, 3})$ by attaching $g$ handles along the (same) disks $\Delta_{k}$, $1\leq k \leq g$. Then $f$ induces a proper map, which we also call $f$, from $\Sigma_2$ to $\Sigma_1$. By \Cref{degreeonemapchecking}, $\deg(f)=\pm 1$. 

Further, we claim that $f\co\Sigma_2\to\Sigma_1$ is not injective on $\pi_1$. Namely, the fundamental group of $\Sigma_2$ is the amalgamated free product of four groups, one of which is $\pi_1(\mathbb S^2\setminus \textup{int}(\mathcal D_0\cup\mathcal D_{1, 1}\cup\mathcal D_{1,2}\cup\mathcal D_{1,3}))$. 
As $\gamma$ is not homotopic to the trivial loop in $\mathbb S^2\setminus \textup{int}(\mathcal D_0\cup\mathcal D_{1, 1}\cup\mathcal D_{1,2}\cup\mathcal D_{1,3})$, and components of an amalgamated free product inject, $\gamma$ is not homotopic to the trivial loop in $\Sigma_2$. However, $f(\gamma)$ is homotopic to the trivial loop in $\mathbb S^2\setminus \textup{int}(\mathcal D_0\cup\mathcal D_1)$ and hence in $\Sigma_1$. Therefore, $f$ is not injective on $\pi_1$.

\begin{figure}[ht]$$\adjustbox{trim={0.0\width} {0.0\height} {0.13\width} {0.0\height},clip}{\def\svgwidth{1\linewidth}
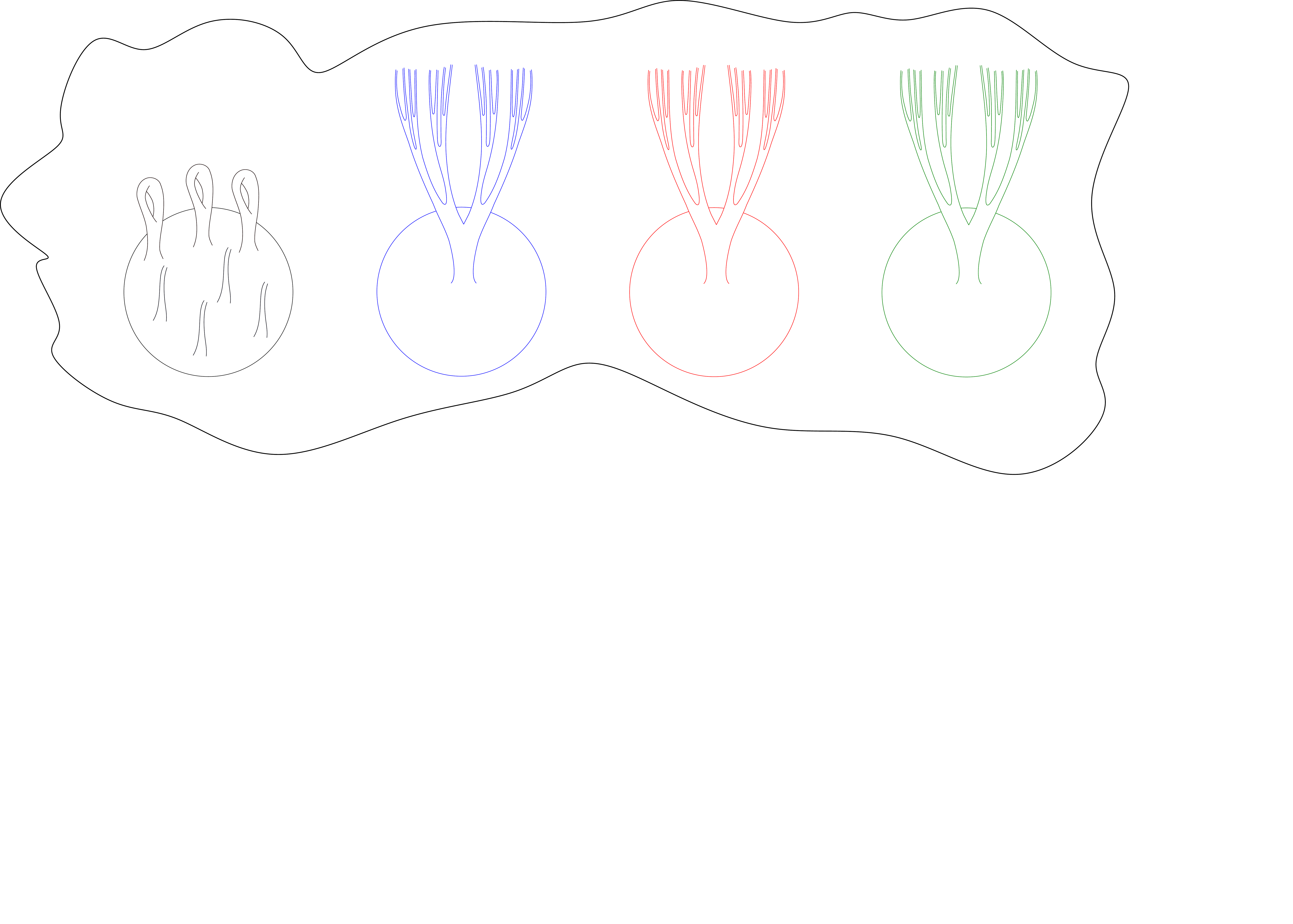
}$$
\caption{A non $\pi_1$-injective degree $\pm 1$ map $f\colon \Sigma\to \Sigma$, where $g=3$ and $|\mathscr I|=4$.}\label{cantortree}

\end{figure}

Both $\Sigma_1$ and $\Sigma_2$ have genus the same as $\Sigma$, and the space of ends homeomorphic to that of $\Sigma$ (as a finite disjoint union of Cantor spaces is a Cantor space by the universality of the Cantor set) with all ends planar. Hence, by \Cref{richard2}, both $\Sigma_1$ and $\Sigma_2$ are homeomorphic to $\Sigma$.

Identifying $\Sigma_1$ and $\Sigma_2$ with $\Sigma$ by homeomorphisms, we get a proper map $f\co\Sigma\to\Sigma$ which is not injective on $\pi_1$. As homeomorphisms have degree $\pm 1$, it follows that $\deg(f)=\pm 1$. Replacing $f$ by $f\circ f$ if necessary, we obtain a proper map of degree one  that is not injective on $\pi_1$.
\end{proof}

\begin{theorem}
    \textup{Let $\Sigma$ be a finite genus surface  such that $\mathscr E(\Sigma)$ has infinitely many isolated points. Then there is a degree one map $f\co \Sigma\to \Sigma$ which is not $\pi_1$-injective.}\label{finitegenswithINfinitelymanyisolatedends}
\end{theorem}

\begin{proof}
Let $\mathscr I(\Sigma)$ be the set of all isolated points of $\mathscr E(\Sigma)$, and let $g$ be the genus of $\Sigma$. Also, let $\mathcal D_0, \mathcal D_1, \mathcal D_{1,1},\mathcal D_{1, 2},\mathcal D_{1, 3}\subseteq \mathbb S^2$, $f\co\mathbb S^2\to \mathbb S^2$, and $\gamma$ be as in the conclusion of \Cref{foldingmapconstruction}. Now, consider a subset $\mathscr E$ of $\text{int}(\mathcal D_0)$ such that $\mathscr E$ is homeomorphic to $\mathscr E(\Sigma)$. Also, consider points $p_1\in \text{int}(\mathcal D_1)$ and $p_{1, i}\in\text{int}(\mathcal D_{1, i})$, $i=1, 2, 3$ such that $f(p_{1,i})=p_1$ for each $i=1,2,3$. See \Cref{flutedpattern}.

Recall that  $f^{-1}(\mathcal D_0)=\mathcal D_0$ and $f\vert_{\mathcal D_0}\co \mathcal{D}_0\to \mathcal D_0$ is the identity map. Thus $f^{-1}(\mathscr E)=\mathscr E$. Now, let $\Sigma_1$ be the surface obtained from $\mathbb S^2\setminus (\mathscr E\cup\{p_1\})$ by attaching $g$ handles along disjoint disks $\Delta_{k}\subset \text{int}(\mathcal D_0)\setminus \mathscr I$, $1\leq k \leq g$, and let $\Sigma_2$ be the surface obtained from $\mathbb S^2\setminus (\mathscr E\cup \{p_{1,1}, p_{1, 2}, p_{1, 3}\})$ by attaching $g$ handles along the same disks $\Delta_{k}$, $1\leq k \leq g$. Then $f$ induces a proper map, which we also call $f$, from $\Sigma_2$ to $\Sigma_1$. By \Cref{degreeonemapchecking}, $\deg(f)=\pm 1$.

\begin{figure}[ht]$$\adjustbox{trim={0.0\width} {0.0\height} {0.07\width} {0.0\height},clip}{\def\svgwidth{1\linewidth}
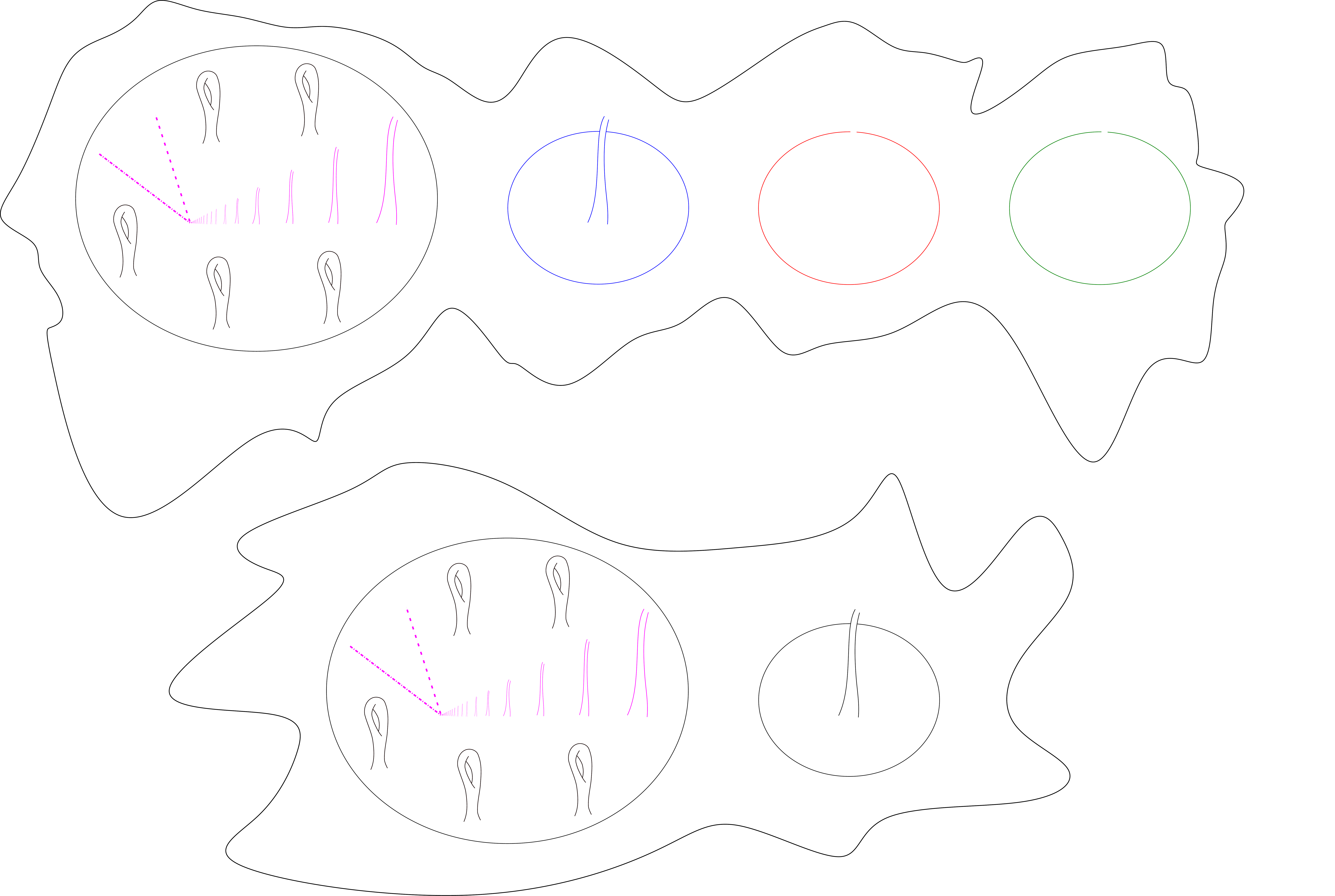
}$$
\caption{A non $\pi_1$-injective degree $\pm 1$ map $f\colon \Sigma\to \Sigma$, where $g=5$ and $\mathscr I$ is an infinite set.}\label{flutedpattern}

\end{figure}

Further, we claim that $f\co\Sigma_2\to\Sigma_1$ is not injective on $\pi_1$. Namely, the fundamental group of $\Sigma_2$ is the amalgamated free product of four groups, one of which is $\pi_1(\mathbb S^2\setminus \textup{int}(\mathcal D_0\cup \mathcal D_{1, 1}\cup\mathcal D_{1,2}\cup\mathcal D_{1,3}))$. 
As $\gamma$ is not homotopic to the trivial loop in $\mathbb S^2\setminus \textup{int}(\mathcal D_0\cup\mathcal D_{1, 1}\cup\mathcal D_{1,2}\cup\mathcal D_{1,3})$, and components of an amalgamated free product inject, $\gamma$ is not homotopic to the trivial loop in $\Sigma_2$. However, $f(\gamma)$ is homotopic to the trivial loop in $\mathbb S^2\setminus \textup{int}(\mathcal D_0\cup\mathcal D_1)$ and hence in $\Sigma_1$. Therefore, $f$ is not injective on $\pi_1$.

Both $\Sigma_1$ and $\Sigma_2$ have genus the same as $\Sigma$ and, by \Cref{add-finite} below, $\mathscr E(\Sigma_1)$ and $\mathscr E(\Sigma_2)$ are homeomorphic to $\mathscr E(\Sigma)$. Further, all ends of $\Sigma$, $\Sigma_1$ and $\Sigma_2$ are planar. Hence, by  \Cref{richard2} both $\Sigma_1$ and $\Sigma_2$ are homeomorphic to $\Sigma$.

Identifying $\Sigma_1$ and $\Sigma_2$ with $\Sigma$ by homeomorphisms, we get a proper map $f\co\Sigma\to\Sigma$ which is not injective on $\pi_1$. As homeomorphisms have degree $\pm 1$, it follows that $\deg(f)=\pm 1$. Replacing $f$ by $f\circ f$ if necessary, we obtain a proper map of degree one that is not injective on $\pi_1$.
\end{proof}

\begin{lemma}
  \textup{Let $\mathscr E$ be a closed totally disconnected subset of $\mathbb S^2$. Let $\mathscr I$ be the set of all isolated points of $\mathscr E$. Assume $\mathscr I$ is infinite. If $\mathscr F$ is a finite subset of $\mathbb S^2\setminus\mathscr E$, then  $\mathscr E\cup \mathscr F$ is homeomorphic to $\mathscr E.$} \label{add-finite}
\end{lemma}

 \begin{proof}
Let $\mathscr A\coloneqq\{a_1,a_2,...\}$ be a subset of $\mathscr I$ such that $a_n\to \ell\in \mathscr E$ ($\mathscr{A}$ exists as $\mathscr E$ is compact and infinite). Define $\mathscr B\coloneqq\mathscr A\cup \mathscr F$. Write $\mathscr B$ as $\mathscr B= \{b_1,b_2,...\}$. Then the map $g\co \mathscr E\cup \mathscr F\to\mathscr E$ defined by $$g(z)\coloneqq\begin{cases}z&\text{ if }z\in (\mathscr E\cup \mathscr F)\setminus \mathscr B,\\ a_n&\text{ if }z=b_n\in \mathscr B,\end{cases}$$ is a homeomorphism. To prove this, note that $g$ is a bijection from a compact space to a Hausdorff space, so it suffices to show that $g$ is continuous. But observe that $g$ restricted to the closed set $(\mathscr E\cup \mathscr F)\setminus \mathscr B$ is the identity, so $g$ is continuous on $(\mathscr E\cup \mathscr F)\setminus \mathscr B$. Also $g$ restricted to the closed set $\mathscr B\cup \{\ell\}$ is continuous as $b_n\to \ell$ and $g(b_n)=a_n\to \ell = g(\ell)$, and all other points of $\mathscr B\cup \{\ell\}$ are isolated. Thus $g$ is continuous, as required.
 \end{proof}

\begin{remark}
\textup{In the paper \cite{MR4353971}, the authors have proved that for every infinite-type surface $\Sigma$, there exists a subsurface homeomorphic to $\Sigma$ such that the inclusion map is not homotopic to a homeomorphism. As our surfaces are connected, this type of inclusion map can’t be proper because of the following two facts:}
\begin{itemize}
    \item \textup{Any injective map between two boundaryless topological manifolds of the same dimension is an open map. This follows from the invariance of domain.}
    \item \textup{Any proper map between two topological manifolds is a closed map, as manifolds are compactly generated spaces, see \cite{MR254818}.}
\end{itemize}\textup{Also, notice that all our results are related to proper maps.}
\end{remark}

\subsection*{Acknowledgements}
The first author is supported by a fellowship from the National Board for Higher Mathematics. We would like to thank Hugo Parlier for telling us about the excellent result of \cite{MR4353971}. We are grateful to the anonymous referee for his careful reading of the paper and his comments and suggestions, which helped considerably in improving the manuscript.

\bibliographystyle{plain}
\bibliography{bibliography.bib}

\newcommand{\Addresses}{{% additional braces for segregating \footnotesize
  \bigskip
  \footnotesize

  \textsc{Department of Mathematics, Indian Institute of Science,
   Bangalore 560012, India}\par\nopagebreak
  \textit{E-mail addresses}: \texttt{\href{mailto:sumantadas@iisc.ac.in}{sumantadas@iisc.ac.in}} and 
  \texttt{\href{mailto:gadgil@iisc.ac.in}{gadgil@iisc.ac.in}}

}}
\Addresses

\end{document}